\documentclass[11pt]{article}
\usepackage{amsmath,amsthm,amssymb}
\usepackage[usenames,dvipsnames]{xcolor}
\usepackage{enumerate}
\usepackage{graphicx}
\usepackage{cite}
\usepackage{comment}
\usepackage{oands}
\usepackage{tikz}
\usepackage{changepage}
\usepackage{bbm}
\usepackage{mathtools}
\usepackage[margin=1in]{geometry}
\usepackage[pagewise,mathlines]{lineno}
\usepackage{appendix}
\usepackage{stmaryrd} 
\usepackage{multicol}
\usepackage{microtype}
\usepackage[colorinlistoftodos]{todonotes}

\usepackage[pdftitle={Conformal covariance of the Liouville quantum gravity metric},
  pdfauthor={Ewain Gwynne and Jason Miller},
colorlinks=true,linkcolor=NavyBlue,urlcolor=RoyalBlue,citecolor=PineGreen,bookmarks=true,bookmarksopen=true,bookmarksopenlevel=2,unicode=true,linktocpage]{hyperref}

\theoremstyle{plain}
\newtheorem{thm}{Theorem}[section]

\newtheorem{lem}[thm]{Lemma}
\newtheorem{prop}[thm]{Proposition}

\def\@rst #1 #2other{#1}
\newcommand\MR[1]{\relax\ifhmode\unskip\spacefactor3000 \space\fi
  \MRhref{\expandafter\@rst #1 other}{#1}}
\newcommand{\MRhref}[2]{\href{http://www.ams.org/mathscinet-getitem?mr=#1}{MR#2}}

\theoremstyle{definition}
\newtheorem{defn}[thm]{Definition}
\newtheorem{remark}[thm]{Remark}

\numberwithin{equation}{section}

\newcommand{\dsb}{\begin{adjustwidth}{2.5em}{0pt}
\begin{footnotesize}}
\newcommand{\dse}{\end{footnotesize}
\end{adjustwidth}}

\newcommand{\ssb}{\begin{adjustwidth}{2.5em}{0pt}}
\newcommand{\sse}{\end{adjustwidth}}

\newcommand{\aryb}{\begin{eqnarray*}}
\newcommand{\arye}{\end{eqnarray*}}
\def\alb#1\ale{\begin{align*}#1\end{align*}}
\def\allb#1\alle{\begin{align}#1\end{align}}
\newcommand{\eqb}{\begin{equation}}
\newcommand{\eqe}{\end{equation}}
\newcommand{\eqbn}{\begin{equation*}}
\newcommand{\eqen}{\end{equation*}}

\newcommand{\BB}{\mathbbm}
\newcommand{\ol}{\overline}
\newcommand{\ul}{\underline}
\newcommand{\op}{\operatorname}

\newcommand{\frk}{\mathfrak}
\newcommand{\eqD}{\overset{d}{=}}
\newcommand{\ep}{\varepsilon}
\newcommand{\rta}{\rightarrow}

\newcommand{\wt}{\widetilde}
 
\newcommand{\mcl}{\mathcal}

\newcommand{\bdy}{\partial}
\newcommand{\rng}{\mathring}

\newcommand{\refcoord}{{\hyperref[item-metric-coord]{IV$'$}}}

\let\originalleft\left
\let\originalright\right
\renewcommand{\left}{\mathopen{}\mathclose\bgroup\originalleft}
\renewcommand{\right}{\aftergroup\egroup\originalright}

\title{Conformal covariance of the\\ Liouville quantum gravity metric for $\gamma \in (0,2)$}
 \date{ }
\author{Ewain Gwynne and Jason Miller  \\ {\it University of Cambridge}}

\begin{document}

\maketitle

\begin{abstract}
For $\gamma \in (0,2)$, $U\subset \BB C$, and an instance $h$ of the Gaussian free field (GFF) on $U$, the $\gamma$-Liouville quantum gravity (LQG) surface associated with $(U,h)$ is formally described by the Riemannian metric tensor $e^{\gamma h} (dx^2 + dy^2)$ on $U$.  
Previous work by the authors showed that one can define a canonical metric (distance function) $D_h$ on $U$ associated with a $\gamma$-LQG surface.
We show that this metric is conformally covariant in the sense that it respects the coordinate change formula for $\gamma$-LQG surfaces.  That is, if $U,\widetilde{U}$ are domains, $\phi \colon U \to \widetilde{U}$ is a conformal transformation, $Q=2/\gamma+\gamma/2$, and $\widetilde h = h\circ\phi^{-1} + Q\log|(\phi^{-1})'|$, then $D_h(z,w) = D_{\widetilde{h}}(\phi(z),\phi(w))$ for all $z,w \in U$.  This proves that $D_h$ is intrinsic to the quantum surface structure of $(U,h)$, i.e., it does not depend on the particular choice of parameterization. 
\end{abstract}


\tableofcontents

\section{Introduction}
\label{sec-intro}

\subsection{Overview}
\label{sec-overview}

Fix $\gamma \in (0,2)$, suppose that $U \subseteq \BB C$ is a domain, and let $h$ be an instance of (some form of) the Gaussian free field (GFF) on $U$.  The $\gamma$-\emph{Liouville quantum gravity} (LQG) surface described by $h$ formally corresponds to
\begin{equation}
\label{eqn:lqg_metric}
e^{\gamma h(z)} (dx^2 + dy^2), \quad z=x+iy
\end{equation}
where $dx^2+dy^2$ denotes the Euclidean metric on $U$.  This expression does not make literal sense since $h$ is a distribution and not a function so does not take values at points.  Previously, the volume form associated with~\eqref{eqn:lqg_metric} was constructed by Duplantier and Sheffield in \cite{shef-kpz} using a regularization procedure.  Namely, for each $\epsilon > 0$ and $z \in U$ so that $B_\epsilon(z) \subseteq U$ let $h_\epsilon(z)$ denote the average of $h$ on $\partial B_\epsilon(z)$.  Then the volume form $\mu_h$ is given by the limit as $\epsilon \to 0$ of
\begin{equation}
\label{eqn:lqg_measure}
\epsilon^{\gamma^2/2} e^{\gamma h_\epsilon(z)} dx dy, \quad z=x+iy
\end{equation}
where $dx dy$ denotes Lebesgue measure on $U$.  The factor $\epsilon^{\gamma^2/2}$ is necessary for the limit to exist and be non-trivial.  It is also possible to use a similar procedure to make sense of the lengths of certain types of curves \cite{shef-kpz,shef-zipper}. See~\cite{kahane,rhodes-vargas-review} for a more general theory of random measures of this type. 

The LQG measure satisfies a certain change of coordinates formula~\cite[Proposition 2.1]{shef-kpz}.  Suppose that $\wt{U} \subseteq \BB C$ is another domain, $\phi \colon U \to \wt{U}$ is a conformal transformation, and
\begin{equation}
\label{eqn:lqg_coord_change}
\wt{h} = h \circ \phi ^{-1} + Q\log |(\phi ^{-1})'|, \quad Q = \frac{2}{\gamma} + \frac{\gamma}{2}	.
\end{equation}
Then a.s.\ $\mu_h(A) = \mu_{\wt{h}}(\phi(A))$ for all Borel sets $A \subseteq U$.  Two domain/field pairs $(U,h)$, $(\wt{U},\wt{h})$ are said to be \emph{equivalent as LQG surfaces} if they are related as in~\eqref{eqn:lqg_coord_change}.  An \emph{LQG surface} is an equivalence class of domain/field pairs with respect to this equivalence relation. We think of two equivalent pairs as being two embeddings of the same surface.

In a previous series of papers~\cite{lqg-tbm1,lqg-tbm2,lqg-tbm3}, a metric (distance function) associated with a $\sqrt{8/3}$-LQG surface was constructed in the special case when $\gamma=\sqrt{8/3}$. These works also showed that a certain special $\sqrt{8/3}$-LQG surface is equivalent, as a metric measure space, to the Brownian map of Le Gall \cite{legall-uniqueness} and Miermont \cite{miermont-brownian-map}. 

This work is part of a larger project which is focused on constructing for all $\gamma\in (0,2)$ the metric space structure of $\gamma$-LQG, i.e., the Riemannian distance function associated with~\eqref{eqn:lqg_metric}, and proving its basic properties.
We now explain the construction of the metric, which was carried out in the previous works~\cite{dddf-lfpp,local-metrics,lqg-metric-estimates,gm-confluence,gm-uniqueness}. 
It is shown in~\cite{dg-lqg-dim,dzz-heat-kernel} that for each $\gamma \in (0,2)$, there is an exponent $d_\gamma > 2$ which can be defined in several equivalent ways, e.g., as the ball volume growth exponent for certain random planar maps in the $\gamma$-LQG universality class. 
It is shown in~\cite{gp-kpz} that $d_\gamma$ is the Hausdorff dimension of a $\gamma$-LQG surface, viewed as a metric space.
The value of $d_\gamma$ is not known explicitly except that $d_{\sqrt{8/3}}=4$, but see~\cite{dg-lqg-dim,gp-lfpp-bounds,ang-discrete-lfpp} for reasonably sharp bounds on $d_\gamma$. 

We define
\eqb \label{eqn-xi-def}
\xi = \xi_\gamma := \frac{\gamma}{d_\gamma} .
\eqe 
The significance of the parameter $\xi$ is as follows: for a smooth function $f$, the Riemannian distance function associated with the metric tensor $e^f (dx^2 + dy^2)$ is obtained by integrating $e^{f/2}$ with respect to the Euclidean length measure on smooth paths. This makes it so that scaling the volume form by a factor of $C$ corresponds to scaling distances by a factor of $C^{1/2}$. The Hausdorff dimension of the $\gamma$-LQG metric is $d_\gamma$, rather than 2, so scaling the volume form (i.e., the LQG area measure) by a factor of $C$ should correspond to scaling distances by a factor of $C^{1/d_\gamma}$. This is achieved by defining the distance function using a regularization of $e^{\xi h}$ instead of $e^{\gamma h / 2}$.

Suppose for simplicity that $h$ is a whole-plane GFF. 
In light of the preceding paragraph, a natural way to approximate the distance function associated with~\eqref{eqn:lqg_metric} is via the random metrics
\eqb \label{eqn-lfpp}
D_h^\ep(z,w) := \inf_{P : z\rta w} \int_0^1 e^{\xi h_\ep(P(t))} |P'(t)| \, dt ,
\eqe
where the infimum is over all piecewise continuously differentiable paths from $z$ to $w$ and $\{h_\ep\}_{\ep > 0}$ is a certain family of continuous functions which approximate the GFF as $\ep\rta 0$ (for technical reasons convergence has only been shown when we take $h_\ep$ to be the convolution of $h$ with the heat kernel).  

It is shown in~\cite{dddf-lfpp} that the family of random metrics~\eqref{eqn-lfpp} (suitably re-scaled) is tight w.r.t.\ the local uniform topology on $\BB C \times \BB C$, and every possible subsequential limit is a metric which induces the Euclidean topology.
See also~\cite{ding-dunlap-lqg-fpp,df-lqg-metric,ding-dunlap-lgd} for earlier tightness results for approximations of the LQG metric, preceding~\cite{dddf-lfpp}. 

Subsequently, it was shown in~\cite{gm-uniqueness}, building on~\cite{local-metrics,lqg-metric-estimates,gm-confluence}, that the subsequential limit is unique, and in fact the metrics $D_h^\ep$, suitably re-scaled, converge in probability as $\ep\rta 0$ to a metric $D_h$ on $\BB C$. This $D_h$ is defined to be the $\gamma$-LQG metric. 
The metric $D_h$ is characterized by a list of axioms including the metric version of the coordinate change formula~\eqref{eqn:lqg_coord_change} for all complex affine functions. These conditions are listed just below. In particular, the metric for $\gamma=\sqrt{8/3}$ is the same as the one in~\cite{lqg-tbm1,lqg-tbm2,lqg-tbm3}. By the local dependence of $D_h$ on $h$, it follows that one can measurably associate an LQG metric for $\gamma \in (0,2)$ with the GFF on any planar domain (see~\cite[Remark 1.5]{gm-uniqueness}).  

The purpose of this work is to show that the resulting metric satisfies the metric analog of~\eqref{eqn:lqg_coord_change} for general conformal maps.  Consequently, the metric constructed in~\cite{gm-uniqueness} is intrinsic to the quantum surface structure of an LQG surface, i.e., the particular choice of embedding does not change the definition of the metric.  As we will see, establishing~\eqref{eqn:lqg_coord_change} for general conformal maps from the case of just complex affine maps is trickier than one might expect. 

Although this work builds on~\cite{dddf-lfpp,local-metrics,lqg-metric-estimates,gm-confluence,gm-uniqueness}, it can be read without any knowledge of these works, or even any knowledge about LQG beyond basic properties of the GFF. 
The reason for this is that we take the axiomatic definition of the whole-plane $\gamma$-LQG metric from~\cite{gm-uniqueness} as our starting point, and deduce our results from these axioms.

\bigskip

\noindent\textbf{Acknowledgements.} We thank two anonymous referees for helpful comments on an earlier version of this paper. We thank Jian Ding, Julien Dub\'edat, Alex Dunlap, Hugo Falconet, Josh Pfeffer, Scott Sheffield, and Xin Sun for helpful discussions. 
EG was supported by a Herchel Smith fellowship and a Trinity College junior research fellowship.
JM was supported by ERC Starting Grant 804166.

\subsection{Main results}
\label{sec-main-results}

We will now define a notion of a $\gamma$-LQG metric for arbitrary open domains $U\subset \BB C$.
The definition of the $\gamma$-LQG metric in~\cite{gm-uniqueness} is the special case when $U = \BB C$. 
We first need some preliminary definitions. Throughout, $(X,D)$ denotes a metric space.
\medskip

\noindent
For a continuous curve $P : [a,b] \rta X$ (here $[a,b]$ is equipped with the Euclidean metric and $X$ is equipped with the metric $D$), the \emph{$D$-length} of $P$ is defined by 
\eqbn
\op{len}\left( P ; D  \right) := \sup_{T} \sum_{i=1}^{\# T} D(P(t_i) , P(t_{i-1})) 
\eqen
where the supremum is over all partitions $T : a= t_0 < \dots < t_{\# T} = b$ of $[a,b]$. Note that the $D$-length of a curve may be infinite.
\medskip

\noindent
For $Y\subset X$, the \emph{internal metric of $D$ on $Y$} is defined by
\eqb \label{eqn-internal-def}
D(x,y ; Y)  := \inf_{P \subset Y} \op{len}\left(P ; D \right) ,\quad \forall x,y\in Y 
\eqe 
where the infimum is over all paths $P$ in $Y$ from $x$ to $y$. 
Then $D(\cdot,\cdot ; Y)$ is a metric on $Y$, except that it is allowed to take the value $\infty$.
\medskip
 
\noindent
We say that $(X,D)$ is a \emph{length space} if for each $x,y\in X$ and each $\ep > 0$, there exists a curve of $D$-length at most $D(x,y) + \ep$ from $x$ to $y$. 
\medskip

\noindent
A \emph{continuous metric} on an open domain $U\subset\BB C$ is a metric $D$ on $U$ which induces the Euclidean topology on $U$. 
We equip the space of such metrics with the local uniform topology for functions from $U\times U$ to $[0,\infty)$. 
We allow a continuous metric to satisfy $D(u,v) = \infty$ if $u$ and $v$ are in different connected components of $U$.
In this case, in order to have $D^n\rta D$ w.r.t.\ the local uniform topology we require that for large enough $n$, $D^n(u,v) = \infty$ if and only if $D(u,v)=\infty$.
\medskip

\noindent
A \emph{GFF plus a continuous function} on an open domain $U\subset \BB C$ is a random distribution $h$ on $U$ which can be coupled with a random continuous function $f$ in such a way that $h-f$ has the law of the (zero-boundary or whole-plane, as appropriate) GFF on $U$. We emphasize that $f$ is not required to extend continuously to $\ol U$.
\medskip
 
\noindent
For $U\subset \BB C$, let $\mcl D'(U)$ be the space of distributions (in the sense of Schwartz) on $\BB C$, equipped with the usual weak topology.

\begin{defn} \label{def-lqg-metric}
A \emph{$\gamma$-LQG metric} is a collection of functions $h\mapsto D_h$, one for each open set $U\subset\BB C$, from $\mcl D'(U)$ to the space of continuous metrics on $U$ with the following properties.
Let $U\subset \BB C$ and let $h$ be a GFF plus a continuous function on $U$.\footnote{Our axioms for a $\gamma$-LQG metric only concern a.s.\ properties of $D_h$ when $h$ is a GFF plus a continuous function. So, once we have defined $D_h$ a.s.\ when $h$ is a GFF plus a continuous function, we can take $D$ to be any measurable mapping $\mcl D'(U) \rta \{\text{continuous metrics on $U$}\}$ which is a.s.\ consistent with our given definition when $h$ is a GFF plus a continuous function. 
In fact, the construction of the metric in~\cite{dddf-lfpp,lqg-metric-estimates,gm-confluence,gm-uniqueness} only gives an explicit definition of $D_h$ in the case when $h$ is a GFF plus a continuous function. }
Then the associated metric $D_h$ satisfies the following axioms.
\begin{enumerate}[I.]
\item \textbf{Length space.} Almost surely, $(U,D_h)$ is a length space, i.e., the $D_h$-distance between any two points of $U$ is the infimum of the $D_h$-lengths of $D_h$-continuous paths (equivalently, Euclidean continuous paths) in $U$ between the two points. \label{item-metric-length}
\item \textbf{Locality.} Let $V \subset U$ be a deterministic open set. 
The $D_h$-internal metric $D_h(\cdot,\cdot ; V)$ is a.s.\ equal to $D_{h|_V}$ (so in particular it is a.s.\ determined by $h|_V$).  \label{item-metric-local}
\item \textbf{Weyl scaling.} Let $\xi$ be as in~\eqref{eqn-xi-def}. For a continuous function $f : U\rta \BB R$, define \label{item-metric-f}
\eqb \label{eqn-metric-f}
(e^{\xi f} \cdot D_h) (z,w) := \inf_{P : z\rta w} \int_0^{\op{len}(P ; D_h)} e^{\xi f(P(t))} \,dt , \quad \forall z,w\in U,
\eqe 
where the infimum is over all continuous paths from $z$ to $w$ in $U$ parameterized by $D_h$-length.
Then a.s.\ $ e^{\xi f} \cdot D_h = D_{h+f}$ for every bounded continuous function $f: U\rta \BB R$.
\item \textbf{Conformal coordinate change.} Let $\wt U\subset \BB C$ and let $\phi : U \rta \wt U$ be a deterministic conformal map. Then with $Q$ as in~\eqref{eqn:lqg_coord_change}, a.s.\ \label{item-metric-coord0}
\eqb \label{eqn-metric-coord0}
 D_h \left( z,w \right) = D_{h\circ\phi^{-1} + Q\log |(\phi^{-1})'|}\left(\phi(z) , \phi(w) \right)  ,\quad  \forall z,w \in U.
\eqe    
\end{enumerate}
\end{defn}

It is shown in~\cite[Theorem 1.2 and Corollary 1.3]{gm-uniqueness} (building on~\cite{dddf-lfpp,local-metrics,lqg-metric-estimates,gm-confluence}) that there is a unique LQG metric in the special case when $U=\BB C$, i.e., one has the following statement.

\begin{thm}[\!\cite{gm-uniqueness}] \label{thm-uniqueness}
There is a measurable function $h\mapsto D_h$ from $\mcl H'(\BB C)$ to the space of continuous metrics on $\BB C$ which satisfies the above axioms with $U = \wt U = \BB C$. In this restricted setting, Axiom~\ref{item-metric-local} is replaced by the requirement that $D_h(\cdot,\cdot;V)$ is a.s.\ determined by $h|_V$ and Axiom~\ref{item-metric-coord0} reads as follows. 
\begin{enumerate}[I$'$.]
\setcounter{enumi}{3} 
\item \textbf{Coordinate change for complex affine maps.} For each fixed deterministic $a,b \in\BB C$, $a\not=0$, a.s.\ \label{item-metric-coord}
\eqb
 D_h \left( a z + b , a w + b \right) = D_{h(a\cdot + b)  +Q\log |a| }(z,w)  , \quad \forall z,w \in \BB C.
\eqe    
\end{enumerate}
Furthermore, if $h\mapsto D_h$ and $h\mapsto \wt D_h$ are two such measurable functions, then there is a deterministic constant $C>0$ such that a.s.\ $D_h = C \wt D_h$ whenever $h$ is a GFF plus a continuous function.
\end{thm}

We call $D_h$ from Theorem~\ref{thm-uniqueness} the \emph{$\gamma$-LQG metric associated with $h$}. 
Following~\cite[Remark 1.2]{gm-confluence}, it is not hard to extend the definition of $D_h$ from Theorem~\ref{thm-uniqueness} to GFF-type distributions on proper subdomains of $\BB C$, as we now explain. 
Suppose that $h$ is a whole-plane GFF. For each deterministic open set $U\subset \BB C$, the metric $D_h(\cdot,\cdot;U)$ is a.s.\ determined by $h|_U$ so we can simply define $D_{h|_U} := D_h(\cdot,\cdot ;U)$.
We can write $h|_U = \rng h^U + \frk h^U$, where $\rng h^U$ is a zero-boundary GFF on $U$ and $\frk h^U$ is a random harmonic function on $U$ independent from $\rng h^U$.
In the notation~\eqref{eqn-metric-f}, we define
\eqb
D_{\rng h^U} := e^{-\xi \frk h^U} \cdot D_{h|_U} .
\eqe
It is easily seen from Axioms~\ref{item-metric-local} and~\ref{item-metric-f} that $D_{\rng h^U}$ is a measurable function of $\rng h^U$ (see~\cite[Remark 1.2]{gm-confluence}).   
In light of Axiom~\ref{item-metric-f}, we can then define $D_{\rng h^U + f}$ as a measurable function of $\rng h^U  + f$ for any random continuous function $f : U\rta\BB R$. 
By inspection, this function from distributions to metrics satisfies Axioms~\ref{item-metric-length} through~\ref{item-metric-f} above. 
The main result of this paper is the following theorem which verifies that the above metric satisfies Axiom~\ref{item-metric-coord0}.  This completes the program to define the $\gamma$-LQG metric for all $\gamma \in (0,2)$ on an arbitrary planar domain.

\begin{thm}
\label{thm-coord}
Let $U\subset\BB C$ be an open domain and let $\phi : \BB C\rta U$ be a conformal map.
Also let $h^U$ be a GFF on $U$ plus a continuous function.
Almost surely, the $\gamma$-LQG metric satisfies the coordinate change formula
\eqb \label{eqn-coord}
D_{h^U}(z,w) = D_{h^U\circ\phi^{-1} + Q\log|(\phi^{-1})'|}\left(\phi(z),\phi(w)\right) ,\quad\forall z,w \in U . 
\eqe
That is, the mapping $h^U \mapsto D_{h^U}$ constructed in~\cite{gm-uniqueness} is a $\gamma$-LQG metric in the sense of Definition~\ref{def-lqg-metric}.
\end{thm}

As noted above, Theorem~\ref{thm-coord} says that the LQG metric depends intrinsically on the $\gamma$-LQG surface $(U,D_h)$, i.e., it does not depend on the particular choice of parameterization for this surface. Hence a $\gamma$-LQG surface with any choice of underlying conformal structure makes sense as a metric space. 

\subsection{Outline}
\label{sec-outline}

Throughout most of the proof of Theorem~\ref{thm-coord}, we will work with a whole-plane GFF $h$ restricted to a domain in $\BB C$. We will transfer to other variants of the GFF at the very end of the argument using Axiom~\ref{item-metric-f}. 

For an open set $U\subset\BB C$ and a conformal map $\phi  : U \rta \phi(U)$, we define 
\eqb \label{eqn-coord-field-def}
h^\phi := h\circ\phi^{-1} +Q\log|(\phi^{-1})'| \quad\text{and} \quad
D_h^\phi(z,w) := D_{h^\phi}(\phi(z),\phi(w)),\quad\forall z,w\in U .
\eqe
By the conformal invariance of the GFF, $h^\phi$ is the sum of a zero-boundary GFF and a harmonic function on $\phi(U)$.  Therefore $D_{h^\phi}$ is defined as explained before the statement of Theorem~\ref{thm-coord}.  
Furthermore, from the locality of $D_h$ it is easily seen that $D_h^\phi$ is a local metric for $h|_U$ and is a.s.\ determined by $h|_U$.
We want to show that a.s.\ $D_h^\phi = D_{h|_U}$. 

As one might expect, the basic idea of the proof is to use that the conformal map $\phi$ looks approximately like a complex affine map in a small neighborhood of a typical point, then apply Axiom~\refcoord. However, there are a number of complications in making this argument work which make the proof of Theorem~\ref{thm-coord} more difficult than one might expect at first glance.
  
The first main step of the proof, which is carried out in \textbf{Section~\ref{sec-bilip}}, is to show that if $z\in U$ and $r >0$ is small, then $D_h$ and $D_h^\phi$ are close on $B_r(z)$, in the sense that with high probability $\sup_{u,v\in B_r(z)} |D_h^\phi(u,v) - D_h(u,v)|$ is of smaller order than the $D_h$-diameter of $B_r(z)$ (which by Axioms~\ref{item-metric-f} and~{\refcoord} is typically of order $r^{\xi Q} e^{\xi h_r(z)}$). 
See Proposition~\ref{prop-scaled-metric-uniform} for a precise statement. 
Here we note that when $r > 0$ is small, the $D_h$-diameter of $B_r(z)$ is smaller than its $D_h$-distance to $\bdy U$, so the restrictions to $B_r(z)$ of $D_{h|_U} = D_h(\cdot,\cdot ; U)$ and $D_h$ agree. 

The main difficulty in this step is that we do not know a priori that $\phi\mapsto D_h^\phi$ depends continuously on the conformal map $\phi$ in the almost sure sense.
This is because we do not know that $\phi \mapsto D_{h\circ\phi^{-1}}$ is continuous. Rather, we only know that if $\phi^{-1}$ is uniformly close to the linear map $z\mapsto\alpha z$ (which will be the case if we start with an arbitrary conformal map $\phi$ and zoom in on a sufficiently small neighborhood of any given point) then the law of $h\circ\phi^{-1}$ is close to the law of $h(\alpha\cdot)$ in the total variation sense (Lemma~\ref{lem-tv-conv}). This tells us that the \emph{marginal laws} of $D_{h\circ\phi^{-1}}$ and $D_{h(\alpha\cdot)}$ are close. 

We will show in Lemma~\ref{lem-field-metric-converge} that \emph{joint law} of $D_{h\circ\phi^{-1}}$ and $D_{h(\alpha\cdot)}$ is close to the joint law of two copies of the same instance of $D_{h(\alpha\cdot)}$. The basic idea of the argument is as follows. If $\{\phi_n\}_{n\in\BB N}$ is a sequence of conformal maps such that $\phi_n^{-1}$ converges uniformly on compact subsets of $\BB C$ to $z\mapsto \alpha z$, then using basic facts about the GFF we can establish the convergence of joint laws
\eqb
(h \circ\phi_n^{-1} , D_{h\circ\phi_n^{-1}}) \rta (h(\alpha\cdot) , D_{h(\alpha\cdot)} ) \quad\text{and} \quad (h , h\circ\phi_n^{-1}) \rta (h , h(\alpha\cdot)).
\eqe
This implies that the joint laws of the 4-tuples $(h , D_h , h\circ\phi_n^{-1} , D_{h\circ\phi_n^{-1}})$ are tight, and moreover allows us to show that any possible subsequential limit is of the form $(h,D_h,h(\alpha\cdot) , D_{h(\alpha\cdot)})$. 

By re-scaling and applying Axiom~\refcoord, the preceding paragraph allows us to show that if $\phi : U\rta \phi(U)$ is a conformal map, then the metric $D_{h|_U}$ and the metric $D_h^\phi $ appearing in~\eqref{eqn-coord} are close at small scales in the desired sense. 

In \textbf{Section~\ref{sec-coord-proof}}, we upgrade from the statement that $D_{h|_U}$ and $D_h^\phi$ are close with high probability in a small neighborhood of any point to the statement that $D_{h|_U}$ and $D_h^\phi$ are close with high probability everywhere. 
This will be carried out in two steps. In Section~\ref{sec-coord-bilip}, we show that $D_{h|_U}$ and $D_h^\phi$ are a.s.\ bi-Lipschitz equivalent using a general criterion for bi-Lipschitz equivalence of two local metrics for the same GFF (Theorem~\ref{thm-bilip}). We then show that the optimal bi-Lipschitz constant is 1 in Section~\ref{sec-attained} using a ``good annulus covering" argument similar to the one in~\cite[Section 3]{gm-uniqueness}.

The reason why we need to use a two-step argument of this form is as follows. Even though we know that $D_{h|_U}$ and $D_h^\phi$ are close at small scales, our estimates are not sharp enough to say directly that a quantity of the form $\sum_{j=1}^N |D_{h|_U}(P(t_{j-1}) , P(t_j)) - D_h^\phi(P(t_{j-1}) , P(t_j))|$ is small when $P : [0,T] \rta U$ is a $D_h$-rectifiable path and $0 = t_0 < t_1 < \dots < t_N = T$ is a fine partition of $[0,T]$. 
The arguments of Section~\ref{sec-coord-proof} allow us to restrict attention to ``good" scales where we can say that the ratios of certain $D_{h|_U}$-distances and $D_h^\phi$-distances are close to 1.

\subsection{Basic notation}
\label{sec-notation}

\noindent
We write $\BB N = \{1,2,3,\dots\}$ and $\BB N_0 = \BB N \cup \{0\}$. 
\medskip

\noindent
For $a < b$, we define the discrete interval $[a,b]_{\BB Z}:= [a,b]\cap\BB Z$. 
\medskip

\noindent
If $f  :(0,\infty) \rta \BB R$ and $g : (0,\infty) \rta (0,\infty)$, we say that $f(\ep) = O_\ep(g(\ep))$ (resp.\ $f(\ep) = o_\ep(g(\ep))$) as $\ep\rta 0$ if $f(\ep)/g(\ep)$ remains bounded (resp.\ tends to zero) as $\ep\rta 0$.  
We similarly define $O(\cdot)$ and $o(\cdot)$ errors as a parameter goes to infinity. 
We will often specify any requirements on the dependencies on rates of convergence in $O(\cdot)$ and $o(\cdot)$ errors in the statements of lemmas/propositions/theorems, in which case we implicitly require that errors, implicit constants, etc., appearing in the proof satisfy the same dependencies.  
\medskip

\noindent
For $z\in\BB C$ and $r>0$, we write $B_r(z)$ for the Euclidean ball of radius $r$ centered at $z$. We also define the open annulus
\eqb \label{eqn-annulus-def}
\BB A_{r_1,r_2}(z) := B_{r_2}(z) \setminus \ol{B_{r_1}(z)} ,\quad\forall 0 < r_r < r_2 < \infty .
\eqe
\medskip

\section{Comparison of $D_h$ and $D_h^\phi$ at small scales}
\label{sec-bilip}

The goal of this section is to show that in the notation~\eqref{eqn-coord-field-def}, the metrics $D_{h^\phi}$ and $D_{h|_U}$ are close with high probability at small scales (see Proposition~\ref{prop-scaled-metric-uniform} just below). 

We will be working with conformal maps, so since circles are not preserved under conformal maps it is sometimes convenient to use a slightly different normalization for the GFF than the usual $h_1(0) = 0$. In particular, we fix a smooth compactly supported, radially symmetric bump function $\BB f : \BB C\rta [0,1]$ with $\int_{\BB C} f(z)\,dz = 1$ and for $z\in\BB C$ and $r>0$ we define
\eqb \label{eqn-smooth-normalize}
h_{\BB f , r}(z) := (h(r\cdot+z) , \BB f) = (h , r^{-2} \BB f(r^{-1}(\cdot-z))) .
\eqe
We will often normalize $h$ by requiring $h_{\BB f , 1}(0) = 0$ instead of $h_1(0) = 0$. 
The advantage of this is that, since $\BB f$ is smooth, the smoothed average $h_{\BB f , 1}(0)$ depends continuously on $h$ in the distributional topology. This fact is needed in the proof of Lemma~\ref{lem-field-metric-converge0} below.

The main result of this section is the following proposition. 

\begin{prop} \label{prop-scaled-metric-uniform}
Let $h$ be a whole-plane GFF normalized so that $h_{\BB f,1}(0) = 0$.  
For each fixed $\delta  > 0$ and compact set $K\subset U$, 
\eqb \label{eqn-scaled-metric-uniform}
\lim_{r\rta 0} \inf_{z\in K} \BB P\left[ \sup_{u,v\in B_r(z)} | D_h^\phi(u,v)  -  D_h(u,v) | \leq \delta r^{\xi Q} e^{\xi h_{\BB f , r}(z)} \right] =1 .
\eqe
\end{prop}

Proposition~\ref{prop-scaled-metric-uniform} involves the smooth averages $h_{\BB f , r}(z)$ instead of circle averages, but it is easy to convert to statements which do not depend on the choice of normalization for the field. For example, we have the following consequence of Proposition~\ref{prop-scaled-metric-uniform} which will be used in Section~\ref{sec-attained}. 

\begin{lem} \label{lem-scaled-metric-ratio}
Let $h$ be a whole-plane GFF, with any choice of normalization.
Fix $\delta , b \in (0,1)$. For each fixed compact set $K\subset U$, 
\eqb \label{eqn-scaled-metric-ratio}
\lim_{r\rta 0} \inf_{z\in K} \BB P\left[  1-\delta \leq \frac{D_h^\phi(u,v)}{D_h(u,v)} \leq 1+\delta ,\: \forall u,v\in B_r(z) \: \text{with}\: |u-v|\geq b r \right] = 1 .  
\eqe
\end{lem}
\begin{proof}
By Axiom~\ref{item-metric-f}, changing the normalization of $h$ (i.e., adding a constant to $h$) does not affect the value of $D_h^\phi(u,v) / D_h(u,v)$. 
Therefore, we can assume without loss of generality that $h$ is normalized so that $h_{\BB f , 1}(0) = 0$. 
By Axioms~\ref{item-metric-f} and~\refcoord\, and the scale and translation invariance of the law of $h$, modulo additive constant, for each $\ep > 0$ we can find $c > 0$ depending only on $b$ such that with probability at least $1-\ep$,
\eqbn
 D_h(u,v) \geq c r^{\xi Q} e^{\xi h_{\BB f  ,r}(z)} ,\: \forall u,v\in B_r(z) \: \text{with}\: |u-v|\geq b r . 
\eqen
The lemma statement follows by combining this with Proposition~\ref{prop-scaled-metric-uniform} with $c\delta$ in place of $\delta$, then sending $\ep\rta 0$.
\end{proof}

\subsection{$D_h^\phi$ converges to $D_h$ as $\phi$ converges to a linear map}
\label{sec-field-metric-converge}

Throughout this section we let $h$ be a whole-plane GFF normalized so that $h_{\BB f , 1}(0) = 0$, with $\BB f$ as in~\eqref{eqn-smooth-normalize}.
The main step in the proof of Proposition~\ref{prop-scaled-metric-uniform} is the following lemma, which we will prove in this section. 

\begin{lem} \label{lem-field-metric-converge}
Let $\phi_n : U_n\rta\phi_n(U_n)$ be a sequence of conformal maps such that $\phi_n(0) = 0$, $\phi_n'(0) \rta 1/\alpha \in\BB C \setminus \{0\}$ as $n\rta\infty$, and each fixed compact subset of $\BB C$ is contained in $U_n$ for large enough $n$. 
Then
\eqb \label{eqn-field-metric}
\left( h , D_h , h^{\phi_n}  , D_h^{\phi_n} \right) 
\rta \left( h , D_h , h(\alpha \cdot) , D_h \right) 
\eqe
in law with respect to the distributional topology and the local uniform topology on $\BB C\times\BB C$, as appropriate. 
\end{lem}

The main difficulty in the proof of Lemma~\ref{lem-field-metric-converge} is comparing the metrics $D_{h\circ\phi_n^{-1}}$ and $D_{h(\alpha\cdot)}$.  
We will accomplish this using the outline discussed in Section~\ref{sec-outline}. 
We first need the following elementary lemma for the GFF. 

\begin{lem} \label{lem-tv-conv}
Let $\phi_n : U_n\rta\phi_n(U_n)$ be as in Lemma~\ref{lem-field-metric-converge}. 
Then for each compact set $K\subset\BB C$,  
\eqb \label{eqn-as-conv}
 h\circ\phi_n^{-1} |_K \rta h(\alpha\cdot)|_K,\quad\text{almost surely in the distributional sense}
\eqe
and
\eqb \label{eqn-tv-conv}
\left( h\circ\phi_n^{-1}  - (h\circ\phi_n^{-1})_{\BB f , |\alpha| }(0) \right)|_K \rta h(\alpha  \cdot)|_K ,\quad \text{in the total variation sense} .
\eqe  
\end{lem}
\begin{proof}
By the Koebe distortion theorem, $\phi_n'(z) \rta 1/\alpha$ uniformly on compact subsets of $\BB C$. 
It follows that $\phi_n(z) \rta z/\alpha$ uniformly on compact subsets of $\BB C$. 
By the Cauchy integral formula, all of the higher-order derivatives of $\phi_n$ converge to zero uniformly on compact subsets of $\BB C$. 
Furthermore, $\phi_n^{-1} \rta \alpha z$, $(\phi_n^{-1})'(z) \rta \alpha $, and all of the higher-order derivatives of $\phi_n^{-1}$ converge to zero uniformly on compact subsets of $\BB C$. 

Consequently, if $f : \BB C\rta\BB R$ is a smooth, compactly supported function then $f\circ \phi_n$ and all of its derivatives of all orders converge uniformly to $f(\alpha\cdot)$ and its corresponding derivatives as $n\rta\infty$. 
Therefore, 
\eqb
 ( h\circ\phi_n^{-1} , f) = (h , |\phi_n'|^2 (f \circ\phi_n) )  \rta (h ,   |\alpha|^{-2} f(\alpha^{-1} \cdot)  )  = (h(\alpha \cdot) , f) .
\eqe 
This gives~\eqref{eqn-as-conv}. 

To prove~\eqref{eqn-tv-conv}, write $h|_{U_n} = \rng h_n + \frk h_n$, where $\rng h_n$ is a zero-boundary GFF on $U_n$ and $\frk h_n$ is an independent random harmonic function on $U_n$. Then $\rng h_n \circ\phi_n^{-1}$ is a zero-boundary GFF on $\phi(U_n)$. 
By~\cite[Proposition 2.10]{ig4}, we have $\rng h_n\circ\phi_n^{-1}|_K \rta h|_K$ in the total variation sense if we view both $\rng h_n$ and $h$ as being defined modulo a global additive constant. 
The field $h(\alpha \cdot)$ is normalized so that $(h(\alpha \cdot))_{\BB f , |\alpha| }(0)$ is zero. 
Therefore,
\eqb \label{eqn-tv-conv0}
\left(\rng h_n\circ\phi_n^{-1} - (\rng h_n\circ\phi_n^{-1})_{\BB f , |\alpha|}(0)   \right)|_K  \rta h(\alpha \cdot)|_K
\eqe 
in total variation, \emph{without} having to view the distributions as being defined modulo additive constant. 

On the other hand, basic estimates for the harmonic part of the GFF (see the proof of~\cite[Proposition 2.10]{ig4}) combined with the aforementioned convergence of $\phi_n^{-1}$ to $z\mapsto \alpha  z$ shows that for any fixed compact set $K'\subset\BB C$, the Dirichlet energy of $\frk h_n|_{K'}$ tends to zero in probability as $n\rta\infty$. 
Combining this with the convergence of $\phi_n^{-1}$ and all of its derivatives mentioned above, we get that the same is true with $\frk h_n\circ\phi_n^{-1}$ in place of $\frk h_n$. 

Recall that if $f$ is a smooth compactly supported bump function on $U_n$, then the laws of $\rng h_n \circ\phi_n^{-1}$ and $\rng h_n \circ\phi_n^{-1} + f$ are mutually absolutely continuous, and the Radon-Nikodym derivative of the latter with respect to the former is $\exp\left( (\rng h_n \circ\phi_n^{-1} ,f)_\nabla - \frac12 (f,f)_\nabla \right)$, where $(g_1,g_2)_\nabla := \frac{1}{2\pi} \int_{U_n} \nabla g_1(z) \cdot\nabla g_2(z) \, d^2 z$ denotes the Dirichlet inner product. 
By applying this formula with $f$ equal to a smooth, compactly supported bump function times $\frk h_n\circ\phi_n^{-1} - (\frk h_n\circ\phi_n^{-1})_{\BB f , |\alpha|} $, we obtain~\eqref{eqn-tv-conv} from~\eqref{eqn-tv-conv0} and the preceding paragraph. 
\end{proof}

We can now establish the convergence of the second two coordinates in Lemma~\ref{lem-field-metric-converge}.

\begin{lem} \label{lem-field-metric-converge0}
Let $\phi_n : U_n\rta\phi_n(U_n)$ be as in Lemma~\ref{lem-field-metric-converge}. 
Then
\eqb \label{eqn-field-metric0}
\left(  h^{\phi_n}  , D_h^{\phi_n} \right) 
\rta \left(  h(\alpha \cdot) + Q\log|\alpha| , D_h \right) 
\eqe
in law with respect to the distributional topology on the first coordinate and the local uniform topology on $\BB C\times\BB C$ on the second coordinate.
\end{lem}
\begin{proof}
Consider a large bounded open set $V\subset \BB C$. Since $D_h(\cdot,\cdot; V)$ is a deterministic functional of $h|_V$ (Axiom~\ref{item-metric-local}), the total variation convergence in Lemma~\ref{lem-tv-conv} implies that 
\eqb \label{eqn-field-metric-tv}
\left( \left( h\circ\phi_n^{-1}  - (h\circ\phi_n^{-1})_{\BB f , |\alpha| }(0) \right) |_V , D_{ h\circ\phi_n^{-1}  - (h\circ\phi_n^{-1})_{\BB f , |\alpha| }(0)  }(\cdot,\cdot ; V) \right) 
\rta \left( h(\alpha \cdot)|_V , D_{h(\alpha \cdot)}(\cdot,\cdot; V) \right) 
\eqe
in the total variation sense. Since the function $\BB f$ of~\eqref{eqn-smooth-normalize} is smooth and compactly supported, we can apply~\eqref{eqn-as-conv} of Lemma~\ref{lem-tv-conv} to get that
\eqb \label{eqn-normalization-to-0}
(h\circ\phi_n^{-1})_{\BB f , |\alpha| }(0) 
\rta (h(\alpha \cdot))_{\BB f , |\alpha| }(0) 
= h_{\BB f , 1}(0) = 0 
\eqe 
in law as $n\rta\infty$. 

By combining~\eqref{eqn-field-metric-tv} and~\eqref{eqn-normalization-to-0} (and using Axiom~\ref{item-metric-f} to get that the map $c\mapsto D_{h+c}$ is continuous), then letting $V$ increase to all of $\BB C$, we obtain 
\eqb \label{eqn-field-metric-law}
\left(  h\circ\phi_n^{-1}    , D_{ h\circ\phi_n^{-1}   }  \right) 
\rta \left( h(\alpha \cdot)  , D_{h(\alpha \cdot)} \right) 
\eqe
in law.
 
Recall that $h^{\phi_n} = h\circ\phi_n^{-1} + Q\log|(\phi_n^{-1})'|$ and $D_h^{\phi_n}  =D_{h^{\phi_n}}(\phi_n(\cdot),\phi_n(\cdot))$.  
We have $\phi_n(z) \rta \alpha^{-1} z$ and $Q\log|(\phi_n^{-1})'| \rta Q\log|\alpha |$ uniformly on compact subsets of $\BB C$ (see the beginning of the proof of Lemma~\ref{lem-tv-conv}). 
Combining this with~\eqref{eqn-field-metric-law} and using Axiom~\ref{item-metric-f} to deal with the convergence of the metrics shows that
\eqb \label{eqn-field-metric-law'}
\left(  h^{\phi_n}   , D_h^{\phi_n}  \right) 
\rta \left( h(\alpha \cdot) + Q\log|\alpha|  , D_{h(\alpha \cdot) + Q\log|\alpha |}(\alpha^{-1} \cdot,\alpha^{-1} \cdot) \right)  
\eqe
in law. The right side of~\eqref{eqn-field-metric-law'} equals $\left( h(\alpha \cdot) + Q\log|\alpha| , D_h \right) $ by Axiom~\refcoord. 
\end{proof}

\begin{proof}[Proof of Lemma~\ref{lem-field-metric-converge}]
By Lemma~\ref{lem-field-metric-converge0} and the Prokhorov theorem, for any sequence of $n$'s tending to $\infty$, there is a subsequence $\mcl N$ and a coupling $\left( h  , D_h  ,h'  , D_{h'} \right)$ of two whole-plane GFF's and their associated metrics such that as $\mcl N \ni n\rta\infty$, 
\eqbn
\left( h , D_h , h^{\phi_n}  , D_h^{\phi_n} \right)  
\rta \left( h , D_h , h'(\alpha \cdot) + Q\log|\alpha| , D_{h'} \right)  ,
\eqen
in law. By the a.s.\ convergence part of Lemma~\ref{lem-tv-conv}, we have $(h , h^{\phi_n})  \rta (h, h(\alpha \cdot) + Q\log|\alpha| ) $ in law. 
Hence $h'=h$ a.s., so also $D_{h'}  =D_h$ a.s. 
Therefore our subsequential limit is given by the right side of~\eqref{eqn-field-metric}.
Since our initial choice of subsequence was arbitrary, we obtain the statement of the lemma.
\end{proof}

\subsection{Uniform comparison of $D_h$ and $D_h^\phi$}
\label{sec-scaled-metric-uniform}

Continue to assume that $h$ is a whole-plane GFF normalized so that $h_{\BB f , 1}(0) = 0$.  To deduce Proposition~\ref{prop-scaled-metric-uniform} from Lemma~\ref{lem-field-metric-converge}, we need to re-scale to convert from a statement about conformal maps at small scales to a statement about conformal maps which are close to linear at constant-order scales; and we need to ensure that the estimate we obtain is uniform over all $z\in U$. 
The re-scaling will be accomplished by means of the following basic calculation.

\begin{lem} \label{lem-coord-scale}
Fix $r > 0$ and $z\in\BB C$ and let $\wt h := h(r\cdot + z) - h_{\BB f ,r}(z)$, so that $\wt h \eqD h$. 
Also let $\wt\phi : r^{-1}(U-z)  \rta r^{-1}(\phi(U) -z)$ be defined by $\wt\phi(w) = r^{-1}(\phi(r w + z) - z)$. 
Then
\eqb
D_{\wt h}^{\wt\phi}( u , v) = r^{-\xi Q} e^{-\xi h_{\BB f , r}(z)} D_h^\phi(r u  + z , r v + z)  ,\quad\forall u ,v\in r^{-1}(U-z) .
\eqe
\end{lem}
\begin{proof}   
Recall the definition of $h^\phi$ from~\eqref{eqn-coord-field-def}. 
We apply Axiom~\ref{item-metric-f} and then Axiom~\refcoord\, to $h^\phi$ to get that for $u,v \in r^{-1}(U-z)$, 
\alb
D_{\wt h}^{\wt\phi}(u,v) 
&= D_{h(   \phi^{-1}(r\cdot + z) )  + Q\log|(\phi^{-1})'(r\cdot + z)|     - h_{\BB f , r}(z) }\left(r^{-1} (\phi(r u + z) - z) , r^{-1} ( \phi(r v + z)  -z) \right) \notag\\ 
&= e^{-\xi h_{\BB f , r}(z)} D_{h^\phi(r\cdot + z)}\left(r^{-1} (\phi(r u + z) - z) , r^{-1} ( \phi(r v + z)  -z) \right) \quad \text{(Axiom~\ref{item-metric-f})} \notag\\
&= r^{-\xi Q} e^{-\xi h_{\BB f , r}(z)} D_{h^\phi } \left( \phi(r u + z) , \phi(r v + z) \right)  \quad \text{(Axiom~\refcoord)}  \notag \\
&= r^{-\xi Q} e^{-\xi h_{\BB f , r}(z)} D_h^\phi(r u  + z , r v + z) \quad \text{(by the definition~\eqref{eqn-coord-field-def} of $D_h^\phi$)}.
\ale 
\end{proof}

In what follows, we fix $\delta > 0$ and for $z\in U$ and $r> 0$ such that $B_r(z)\subset U$, we let
\eqb \label{eqn-scaled-metric-event}
F_r(z) := \left\{ \sup_{u,v\in B_r(z)} | D_h^\phi(u,v)  -  D_h(u,v) | \leq \delta r^{\xi Q} e^{\xi h_{\BB f , r}(z)} \right\}
\eqe
be the event of Proposition~\ref{prop-scaled-metric-uniform}.

\begin{lem} \label{lem-scaled-metric-conv}
Consider a sequence of points $\{z_n\}_{n\in\BB N} \subset U$ and radii $\{r_n\}_{n\in\BB N}$ such that $z_n\rta z\in U$ and $r_n\rta 0$.
In the notation of Proposition~\ref{prop-scaled-metric-uniform}, we have $\lim_{n\rta\infty} \BB P[F_{r_n}(z_n)] = 1$. 
\end{lem}
\begin{proof}
For $n\in\BB N$, define the conformal map
\eqbn
\phi_n : r_n^{-1} (U_n - z_n) \rta r_n^{-1}(\phi(U_n) - z_n) \quad \text{by} \quad \phi_n(w) = r_n^{-1} \left( \phi (r_n w + z_n) - z_n \right)
\eqen 
Then $\phi_n(0) = 0$, $\phi_n'(0) = \phi'(z_n) \rta \phi'(z)$, and (since $r_n\rta 0$ and $z$ lies at positive distance from $\bdy U$) every compact subset of $\BB C$ is contained in $r_n^{-1} (U_n - z_n)$ for large enough $n$. 
 
Define the field $h_n := h(r_n \cdot+z_n) - h_{\BB f , r_n}(z_n) \eqD h$. 
By Lemma~\ref{lem-field-metric-converge} applied with $h_n$ in place of $h$ and $\alpha = 1/\phi_n'(z)$, we have the convergence of joint laws
\eqb
\left( h_n , D_{h_n} , h_n^{\phi_n} , D_{h_n}^{\phi_n}  \right) \rta \left( h,D_h,h(\cdot/\phi'(z)) + Q\log |1/\phi'(z)| ,D_h \right) . 
\eqe
In particular, it holds with probability tending to 1 as $n\rta\infty$ that
\eqb \label{eqn-scaled-metric-close}
  | D_{h_n}^{\phi_n}(u,v)  -  D_{h_n}(u,v) | \leq \delta  ,\quad \forall u,v\in \BB D  .
\eqe

By Lemma~\ref{lem-coord-scale} along with Axioms~\ref{item-metric-f} and~\refcoord\, for $h$, if $r_n$ is sufficiently large that $\BB D\subset r_n^{-1}(U-z_n)$, then 
\allb \label{eqn-scaled-metric-compare}
&D_{h_n}(u,v) =  r_n^{-\xi Q} e^{-\xi h_{\BB f , r_n}(z_n)} D_h(r_n u  + z_n , r_n v + z_n) \quad \text{and}  \notag \\ 
&\qquad D_{h_n}^{\phi_n}(u,v) = r_n^{-\xi Q} e^{-\xi h_{\BB f , r_n}(z_n)} D_h^\phi(r_n u  + z_n , r_n v + z_n) ,
\quad \forall u ,v \in \BB D .
\alle
Combining~\eqref{eqn-scaled-metric-close} with~\eqref{eqn-scaled-metric-compare} gives the statement of the lemma.
\end{proof}

\begin{proof}[Proof of Proposition~\ref{prop-scaled-metric-uniform}]
Assume by way of contradiction that for some compact set $K\subset U$, the relation~\eqref{eqn-scaled-metric-uniform} fails.
Then there is an $\ep > 0$, a sequence $r_n\rta 0$, and a sequence of points $z_n \in K$ such that $\BB P[F_{r_n}(z_n)]\leq 1-\ep$ for every $n\in\BB N$. 
By possibly passing to a subsequence, we can assume without loss of generality that $z_n\rta z\in K$. 
Then Lemma~\ref{lem-scaled-metric-conv} shows that $\lim_{n\rta\infty} \BB P[F_{r_n}(z_n)]  = 1$, which is a contradiction. 
\end{proof}

\section{Proof of Theorem~\ref{thm-coord}}
\label{sec-coord-proof}

Recall the notation $D_h^\phi$ from~\eqref{eqn-coord-field-def}. 
To prove Theorem~\ref{thm-coord}, we want to upgrade from the statement that $D_{h|_U}$ and $D_h^\phi$ are close with high probability at small scales (Proposition~\ref{prop-scaled-metric-uniform}) to the statement that these two metrics are a.s.\ close globally. This will be done using various local independence properties of the GFF. 
In this section we will mostly use circle averages rather than the smoothed average $ h_{\BB f , r}(z)$ of Section~\ref{sec-bilip}. 

\subsection{Bi-Lipschitz equivalence of $D_h$ and $D_h^\phi$}
\label{sec-coord-bilip}

Before establishing that $D_{h|_U}  =D_h^\phi$, we will show that the two metrics are bi-Lipschitz equivalent. 

\begin{prop} \label{prop-coord-bilip} 
Let $h$ be a whole-plane GFF, with any choice of additive constant. 
For each conformal map $\phi : U\rta\phi(U)$, there is a constant $C \geq 1$ such that a.s.\ 
\eqb 
C^{-1} D_{h|_U}(z,w) \leq D_h^\phi(z,w) \leq C D_{h|_U}(z,w) ,\quad\forall z,w\in U .
\eqe
\end{prop}

Proposition~\ref{prop-coord-bilip} will be a consequence of Proposition~\ref{prop-scaled-metric-uniform} together with a general criterion for metrics coupled with the same GFF to be bi-Lipschitz equivalent which is proven in~\cite{local-metrics}. To state the criterion, we need a couple of preliminary definitions.

\begin{defn}[Jointly local metrics] \label{def-jointly-local}
Let $U\subset \BB C$ be a connected open set and let $(h,D_1,\dots,D_n)$ be a coupling of a GFF on $U$ and $n$ random continuous length metrics.
We say that $D_1,\dots,D_n$ are \emph{jointly local metrics} for $h$ if for any open set $V\subset U$, the collection of internal metrics $\{ D_j(\cdot,\cdot;  V) \}_{j = 1,\dots,n}$ is conditionally independent from $(h|_{U\setminus V} ,  \{ D_j(\cdot,\cdot;  U\setminus \ol V) \}_{j = 1,\dots,n}   )$ given $h|_V$.
\end{defn}

We note that if each of $D_1,\dots,D_n$ is a local metric for $h$ and is determined by $h$, then $D_1,\dots,D_n$ are jointly local for $h$. 
This is a consequence of~\cite[Lemma 1.4]{local-metrics}.
In particular, $D_{h|_U}$ and $D_h^\phi$ are jointly local for $h|_U$. 

\begin{defn}[Additive local metrics] \label{def-additive-local}
Let $U\subset \BB C$ be a connected open set and let $(h^U ,D_1,\dots,D_n)$ be a coupling of a GFF on $U$ and $n$ random continuous length metrics which are jointly local for $h$.
For $\xi \in\BB R$, we say that $ D_1,\dots,D_n $ are \emph{$\xi$-additive} for $h^U$ if for each $z\in U$ and each $r> 0$ such that $B_r(z) \subset U$, the metrics $(e^{-\xi h_r^U(z)} D_1,\dots, e^{-\xi h_r^U(z)} D_n)$ are jointly local metrics for $h^U - h_r^U(z)$ (where, as per usual, $h^U_r(z)$ denotes the circle average). 
\end{defn}

It is clear from Axiom~\ref{item-metric-f} that the metrics $D_{h|_U}$ and $D_h^\phi$ are $\xi$-additive for $h|_U$. 
The following theorem is~\cite[Theorem 1.6]{local-metrics}.
For the statement, we recall the notation for Euclidean annuli from~\eqref{eqn-annulus-def}. 

\begin{thm} \label{thm-bilip}
Let $\xi \in\BB R$, let $h$ be a whole-plane GFF normalized so that $h_1(0) = 0$, let $U\subset\BB C$, and let $(h,D,\wt D)$ be a coupling of $h$ with two random continuous metrics on $U$ which are jointly local and $\xi$-additive for $h|_U$.  
There is a universal constant $p \in (0,1)$ such that the following is true.
Suppose there is a constant $C>0$ such that for each compact set $K\subset U$, there exists $r_K > 0$ such that
\eqb \label{eqn-bilip}
\BB P\left[  \sup_{u,v \in \bdy B_r(z)} \wt D\left(u,v; \BB A_{r/2,2r}(z) \right) \leq C D(\bdy B_{r/2}(z) , \bdy B_r(z) ) \right] \geq p    ,\quad \forall z\in K, \quad \forall r \in (0,r_K] .
\eqe 
Then a.s.\ $\wt D(z,w) \leq C D(z,w)$ for all $z,w\in\BB C$. 
\end{thm}

Let us now check the condition~\eqref{eqn-bilip} for the metrics $D_{h|_U}$ and $D_{h^\phi}$ using Proposition~\ref{prop-scaled-metric-uniform}.

\begin{lem} \label{lem-coord-ratio}
Let $h$ be a whole-plane GFF, with any choice of additive constant. 
For each $p\in (0,1)$, there exists a constant $C = C(p,\gamma) > 0$ such that for each choice of conformal map $\phi : U \rta \phi(U)$ and each compact set $K\subset U$, there exists $r_K  = r_K(\phi) > 0$ such that  
\eqb
\BB P\left[  \sup_{u,v \in \bdy B_r(z)} D_h^\phi\left(u,v; \BB A_{r/2,2r}(z) \right) \leq C D_{h|_U}(\bdy B_{r/2}(z) , \bdy B_r(z) ) \right] \geq p    ,\quad \forall z\in K, \quad \forall r \in (0,r_K]  ;
\eqe  
and the same is true with $D_h^\phi$ and $D_h$ interchanged. 
\end{lem}
\begin{proof}
As in the proof of Lemma~\ref{lem-scaled-metric-ratio}, due to Axiom~\ref{item-metric-f} we can assume without loss of generality that $h$ is normalized so that $h_{\BB f ,1}(0) = 0$. 
By Axioms~\ref{item-metric-f} and~\refcoord, the fact that $D_h$ induces the Euclidean topology, and the scale and translation invariance of the law of $h$, modulo additive constant, we can find small constants $a,b >0$ such that for each $z\in\BB C$ and each $r>0$, it holds with probability at least $1-(1-p)/2$ that the following is true.
\begin{enumerate}
\item $D_h(\bdy B_r(z) , \bdy \BB A_{r/2,2r}(z)) \geq a  r^{\xi Q} e^{\xi h_{\BB f , r}(z)} $. \label{item-coord-ratio-across}
\item For each $u,v\in \bdy B_r(z)$ with $|u-v| \leq b r$, we have $D_h(u,v) \leq (a/100) r^{\xi Q} e^{\xi h_{\BB f , r}(z)}$. \label{item-coord-ratio-around}
\end{enumerate}
By Proposition~\ref{prop-scaled-metric-uniform}, for each compact set $K\subset U$ there exists $r_K = r_K(\phi) > 0$ such that for each $z\in K$ and each $r\in (0,r_K]$, it holds with probability at least $1-(1-p)/2$ that
\eqb
\sup_{u,v\in B_{2r}(z)} | D_h^\phi(u,v)  -  D_h(u,v) | \leq \frac{a}{100} r^{\xi Q} e^{\xi h_{\BB f , r}(z)} .
\eqe
Combining these estimates shows that for each $z\in K$ and each $r\in (0,r_K]$, it holds with probability at least $p$ that 
\begin{enumerate}
\item $D_h^\phi\left(\bdy B_r(z) , \bdy \BB A_{r/2,2r}(z)  \right) \geq (a/2)  r^{\xi Q} e^{\xi h_{\BB f , r}(z)} $.\label{item-coord-ratio-across'}
\item For each $u,v\in \bdy B_r(z)$ with $|u-v| \leq b r$, we have $D_h^\phi(u,v) \leq (a/2) r^{\xi Q} e^{\xi h_{\BB f , r}(z)}$. \label{item-coord-ratio-around'}
\end{enumerate}

If $u,v\in \bdy B_r(z)$ such that $D_h(u,v) \leq D_h(\bdy B_r(z) , \bdy \BB A_{r/2,2r}(z))$, then $D_h(u,v) = D_h(u,v; \BB A_{r/2,2r}(z))$. 
By applying condition~\ref{item-coord-ratio-around} for $D_h$ and the triangle inequality at most $2\pi/b$ times, we therefore have that
\eqb
\sup_{u,v \in \bdy B_r(z)} D_h \left(u,v; \BB A_{r/2,2r}(z) \right) \leq \frac{a \pi}{50 b} r^{\xi Q} e^{\xi h_{\BB f , r}(z)} .
\eqe
Similarly, by condition~\ref{item-coord-ratio-around'} for $D_h^\phi$ and applying the triangle inequality at most $2\pi/b$ times, we have that  
\eqb
\sup_{u,v \in \bdy B_r(z)} D_h^\phi\left(u,v; \BB A_{r/2,2r}(z) \right) \leq \frac{a \pi}{  b}  r^{\xi Q} e^{\xi h_{\BB f , r}(z)}.
\eqe
Combining the preceding two estimates with condition~\ref{item-coord-ratio-across} for each of $D_h$ and $D_h^\phi$ gives the statement of the lemma with $C = \pi/b$. 
\end{proof}

\begin{proof}[Proof of Proposition~\ref{prop-coord-bilip}]
The metrics $D_h$ and $D_h^\phi$ are each local metrics for $h|_U$. Moreover, these metrics are determined by $h|_U$ so they are jointly local for $h|_U$. By Axiom~\ref{item-metric-f}, for $z\in \BB C$ and $r>0$, the metrics $e^{-\xi h_r(z)} D_h$ and $e^{-\xi h_r(z)} D_h^\phi$ are each local for $h|_U - h_r(z)$, so in particular these metrics are $\xi$-additive for $h|_U$. 
Therefore, the proposition statement follows from Lemma~\ref{lem-coord-ratio} and Theorem~\ref{thm-bilip}.
\end{proof}

\subsection{The bi-Lipschitz constant is 1}
\label{sec-attained}

We will now show that in fact the constant $C$ in Proposition~\ref{prop-coord-bilip} can be taken to be one. 

\begin{prop} \label{prop-attained}
Let $h$ be a whole-plane GFF (with any choice of normalization). Let $U\subset\BB C$ be an open domain and let $\phi : U\rta \phi(U)$ be a conformal map. 
Almost surely, we have $D_h^\phi(z,w) \leq D_{h|_U}(z,w)$ for each $z,w\in U$. 
\end{prop}

Before proving Proposition~\ref{prop-attained}, we explain why it implies our main result.

\begin{proof}[Proof of Theorem~\ref{thm-coord}, assuming Proposition~\ref{prop-attained}]
For a whole-plane GFF $h$ and a domain $U\subset\BB C$, we can write $h|_U$ as the sum of a zero-boundary GFF on $U$ and an independent random harmonic function on $U$. 
Therefore, if $h^U$ is a random distribution on $U$ as in Theorem~\ref{thm-coord}, then we can couple $h^U$ with $h$ in such a way that $g := h|_U -h^U$ is a random continuous function on $U$. 
By Proposition~\ref{prop-attained} and Axiom~\ref{item-metric-f}, it follows that a.s.\ 
\alb
D_{h^U}(z,w)
= (e^{\xi g}\cdot D_{h|_U})(z,w)
&\geq (e^{\xi g} \cdot D_h^\phi)(z,w)  \notag \\
&= (e^{\xi (g\circ\phi^{-1} ) }\cdot D_{h^\phi})\left(\phi(z) , \phi(w) \right) \notag\\
&= D_{h^U\circ\phi^{-1} + Q\log|(\phi^{-1})'|}\left(\phi(z),\phi(w)\right) ,\quad\forall z,w \in U . 
\ale
This is a one-sided version of~\eqref{eqn-coord}. 
By the conformal invariance of the law of the zero-boundary GFF, we can apply this one-sided statement with $h^U\circ\phi^{-1} + Q\log|(\phi^{-1})'|$ in place of $h$ and $\phi^{-1}$ in place of $\phi$ to get the opposite inequality in~\eqref{eqn-coord}. 
\end{proof}

The proof of Proposition~\ref{prop-attained} is similar to the proof of~\cite[Proposition 3.6]{gm-uniqueness}.
Fix a small $\delta \in (0,1)$ (which we will eventually send to zero) and a parameter $\alpha\in (1/2,1)$, a little bit less than 1.
We will use Lemma~\ref{lem-scaled-metric-ratio} together with a general independence result for events for the GFF restricted to annuli (Lemma~\ref{lem-annulus-iterate} just below) to cover a given compact set $K\subset U$ by Euclidean balls of the form $B_{r/2}(z)$ such that $D_h^\phi(u,v) \leq (1+\delta) D_{h|_U}(u,v)$ for each $u \in \bdy B_{\alpha r}(z)$ and each $v\in \bdy B_{r}(z)$ which can be joined by a $D_{h|_U}$-geodesic contained in $\ol{\BB A_{\alpha r , r}(z)}$. 
If we assume that $D_h^\phi \leq C D_{h|_U}$ for some $C>1$, then by considering the times when a $D_{h|_U}$-geodesic between two fixed points $\BB z , \BB w \in \BB C$ crosses the annulus $\BB A_{\alpha r , r}(z)$ for such a $z$ and $r$, we will be able to show that $ D_h^\phi(\BB z , \BB w) \leq C_\delta D_{h|_U}(\BB z,\BB w)$ for a constant $C_\delta$ which is strictly smaller than $C$ if $\delta$ is chosen to be sufficiently small. 
This shows that one has to have  $D_h^\phi \leq C D_{h|_U}$ for $C=1$, since if the optimal constant for which this holds is strictly bigger than 1, then this optimal constant can be improved, which is a contradiction.

The following annulus iteration lemma, which is \cite[Lemma~3.1]{local-metrics}, a generalization of a result from \cite{mq-geodesics}, will be used to produce the desired covering by balls of the form $B_{r/2}(z)$.
 
\begin{lem} \label{lem-annulus-iterate}
Fix $0 < s_1<s_2 < 1$. Let $\{r_k\}_{k\in\BB N}$ be a decreasing sequence of positive real numbers such that $r_{k+1} / r_k \leq s_1$ for each $k\in\BB N$ and let $\{E_{r_k} \}_{k\in\BB N}$ be events such that $E_{r_k} \in \sigma\left( (h-h_{r_k}(0)) |_{\BB A_{s_1 r_k , s_2 r_k}(0)  } \right)$ for each $k\in\BB N$. 
For $K\in\BB N$, let $N(K)$ be the number of $k\in [1,K]_{\BB Z}$ for which $E_{r_k}$ occurs.  
\item For each $a > 0$ and each $b\in (0,1)$, there exists $p = p(a,b,s_1,s_2) \in (0,1)$ and $c = c(a,b,s_1,s_2) > 0$ such that if \label{item-annulus-iterate-high}
\eqb \label{eqn-annulus-iterate-prob}
\BB P\left[ E_{r_k}  \right] \geq p , \quad \forall k \in \BB N,
\eqe 
then 
\eqb \label{eqn-annulus-iterate}
\BB P\left[ N(K)  < b K\right] \leq c e^{-a K} ,\quad\forall K \in \BB N. 
\eqe  
\end{lem}

\begin{figure}[t!]
 \begin{center}
\includegraphics[scale=1]{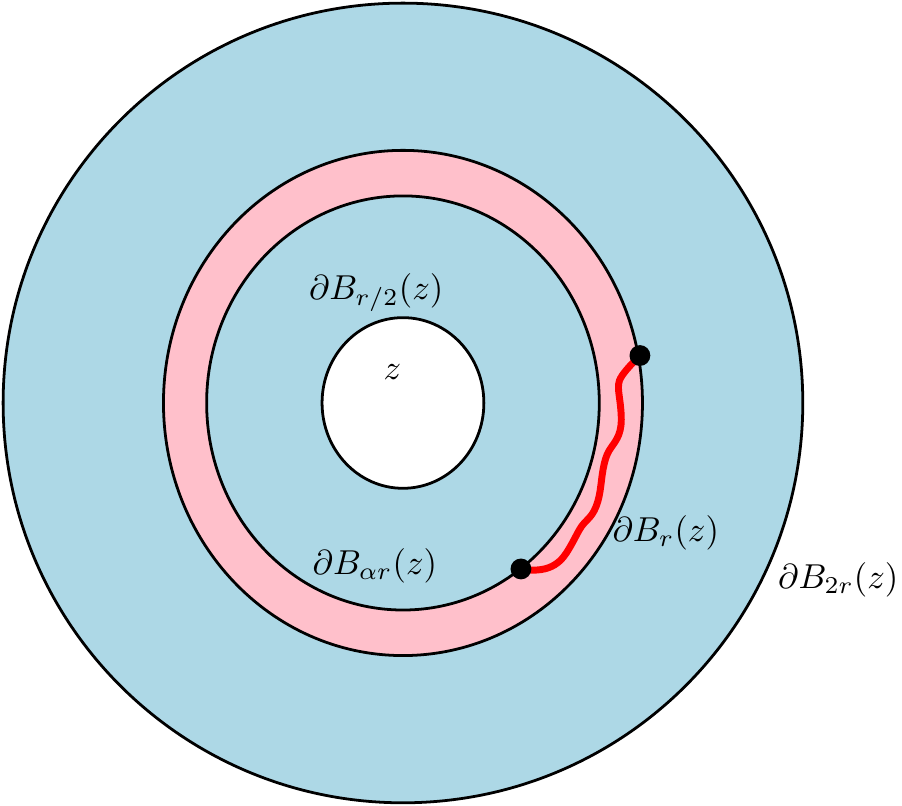}
\vspace{-0.01\textheight}
\caption{
Illustration of condition~\ref{item-attained-long} in the definition of $E_r(z)$. Suppose that $u \in \bdy B_{\alpha r}(z)$ and $v \in \bdy B_{r}(z)$ are ``far apart" in the sense that $v$ is further from $u$ than the boundary of the large blue annulus $\BB A_{r/2,r}(z)$, w.r.t.\ either $D_h$ or $D_h^\phi$. Then condition~\ref{item-attained-long} says that any path from $u$ to $v$ which is contained in $\ol{\BB A_{\alpha r , r}(z)}$ (such as the red path in the figure) has to have $D_h$-length strictly larger than $D_h(u,v;\BB A_{r/2,2r}(z) )$, so in particular such a path cannot be a $D_h$-geodesic. To obtain that this condition holds with high probability, we will choose $\alpha$ to be close to 1 and use the fact paths which are ``close" to a circular arc have large $D_h$-lengths (see Lemma~\ref{lem-attained-long}). 
}\label{fig-attained-long}
\end{center}
\vspace{-1em}
\end{figure}

Let us now define the events to which we will apply Lemma~\ref{lem-annulus-iterate}. 
For $z \in U$, $r > 0$ such that $B_r(z) \subset U$, and parameters $\alpha \in (1/2,1)$, $ A > 1$, and $\delta \in (0,1)$ and let $E_r(z) = E_r(z; \alpha,A,\delta)$ be the event that the following is true.
\begin{enumerate}
\item For each $u \in \bdy B_{\alpha r}(z)$ and each $v\in \bdy B_{r}(z)$ such that there is a $D_h$-geodesic from $u$ to $v$ which is contained in $\ol{\BB A_{\alpha r , r}(z)}$, we have $D^\phi_h(u,v) \leq (1+\delta) D_h(u,v)$. \label{item-attained-dist}
\item If $u \in \bdy B_{\alpha r}(z)$ and $v\in \bdy B_r(z)$ such that either $D_h(u,v) > D_h(u , \bdy\BB A_{r/2,2r}(z))$ or $  D^\phi_h(u,v) > D^\phi_h(u , \bdy\BB A_{r/2,2r}(z))$, then each path from $u$ to $v$ which stays in $\ol{\BB A_{\alpha r , r}(z)}$ has $D_h$-length strictly larger than $D_h\left( u , v ; \BB A_{r/2,2r}(z) \right)$ (see Figure~\ref{fig-attained-long} for an illustration). \label{item-attained-long}
\item There is a path in $\BB A_{\alpha r , r}(z)$ which disconnects the inner and outer boundaries of $\BB A_{\alpha r , r}(z)$ and has $D_h$-length at most $A D_h\left(\bdy B_{\alpha r}(z) , \bdy B_{  r}(z) \right)$. \label{item-attained-around}
\end{enumerate}
Condition~\ref{item-attained-dist} is the main point of the event $E_r(z)$, as discussed just above. 
The purpose of condition~\ref{item-attained-long} is to ensure that $E_r(z)$ is determined by $h|_{\BB A_{r/2,2r}(z)}$. 
Indeed, as we will see in the proof of Lemma~\ref{lem-attained-msrble} just below, on the event that this condition is satisfied $h|_{\BB A_{r/2,2r}(z)}$ determines which paths in $\ol{\BB A_{\alpha r ,r}(z)}$ are $D_h$-geodesics. 
The purpose of condition~\ref{item-attained-around} is to ensure that the annuli $\BB A_{\alpha r , r}(z)$ for which $E_r(z)$ occurs (rather than just the balls $B_r(z)$ for which $E_r(z)$ occurs) cover a positive fraction of the $D_h$-length of a $D_h$-geodesic. 
Indeed, if a $D_h$-geodesic between two points outside of $B_r(z)$ enters $B_{\alpha r}(z)$, then it must cross the path from condition~\ref{item-attained-around} twice. 
Since geodesics are length minimizing, this means that it can spend at most $A D_h\left(\bdy B_{\alpha r}(z) , \bdy B_{  r}(z) \right)$ units of time in $B_{\alpha r}(z)$: otherwise the path from condition~\ref{item-attained-around} would provide a shortcut. 
 
We want to use Lemma~\ref{lem-annulus-iterate} to argue that with high probability we can cover any given compact subset of $U$ by balls $B_{r/2}(z)$ for which $E_r(z)$ occurs. 
We first check the measurability condition in Lemma~\ref{lem-annulus-iterate}

\begin{lem} \label{lem-attained-msrble}
For each $z\in\BB C$ and $r>0$,  
\eqb \label{eqn-attained-msrble}
E_r(z) \in \sigma\left(  (h-h_{4r}(z)) |_{   \BB A_{r/2,2r}(z)  } \right) .
\eqe
\end{lem}
\begin{proof}
By Axiom~\ref{item-metric-f} subtracting $h_{4r}(0)$ from $h$ results in scaling each of $D_h$ and $D^\phi_h$ by $e^{\xi h_{4r}(0)}$, so does not affect the occurrence of $E_r(z)$.  
Hence it suffices to show that $E_r(z) \in   \sigma( h|_{\BB A_{r/2,2r}(z)})$. 
By Axiom~\ref{item-metric-local}, condition~\ref{item-attained-around} in the definition of $E_r(z)$ is determined by $h |_{   \BB A_{r/2,2r}(z)  }$. 

For $u \in \bdy B_{\alpha r}(z)$ and $v\in \bdy B_{r}(z)$, we can determine whether $D_h(u,v) > D_h(u , \bdy\BB A_{r/2,2r}(z))$ from the internal metric $D_h\left(\cdot,\cdot ;\BB A_{r/2,2r}(z)  \right)$: indeed, $D_h(u , \bdy\BB A_{r/2,2r}(z))$ is clearly determined by this internal metric and $D_h(u,v) \leq D_h(u , \bdy\BB A_{r/2,2r}(z))$ if and only if $v$ is contained in the $D_h$-ball of radius $ D_h(u , \bdy\BB A_{r/2,2r}(z))$ centered at $v$, which is contained in $\ol{\BB A_{r/2,2r}(z)}$.  
Similar considerations hold with $D^\phi_h$ in place of $D_h$. 
By the locality of the metrics $D_h$ and $D_h^\phi$, it follows that condition~\ref{item-attained-long} in the definition of $E_r(z)$ is determined by $h|_{\BB A_{r/2,2r}(z)}$. 

If $P$ is a path from $ u \in \bdy B_{\alpha r}(z)$ to $v\in \bdy B_r(z)$ which stays in $\ol{\BB A_{\alpha r ,r}(z)}$, then $P$ is a $D_h$-geodesic if and only if $\op{len}(P ; D_h) = D_h(u,v)$. 
Therefore, if condition~\ref{item-attained-long} holds, then in order for $P$ to be a $D_h$-geodesic we must have $D_h(u,v) \leq D_h(u , \bdy\BB A_{r/2,2r}(z) )$ and $ D^\phi_h(u,v) \leq D^\phi_h(u , \bdy\BB A_{r/2,2r}(z) )$ (note that $D_h(u,v ; \BB A_{r/2,2r}(z)) \geq D_h(u,v)$).
If this is the case, then we can tell whether $P$ is a $D_h$-geodesic from the restriction of $h$ to the $D_h$-metric ball of radius $D_h(u , \bdy\BB A_{r/2,2r}(z) )$ centered at $u$.
We know this restriction is determined by $h|_{\BB A_{r/2,2r}(z)}$ by Axiom~\ref{item-metric-local}. 

On the event that $D_h(u,v) \leq D_h(u , \bdy\BB A_{r/2,2r}(z) )$ and $ D^\phi_h(u,v) \leq D^\phi_h(u , \bdy\BB A_{r/2,2r}(z) )$, both $D_h(u,v)$ and $D^\phi_h(u,v)$ are determined by $h|_{\BB A_{r/2,2r}(z)}$. 
Therefore, the intersection of conditions~\ref{item-attained-dist} and~\ref{item-attained-long} in the definition of $E_r(z)$ is determined by $h|_{\BB A_{r/2,2r}(z)}$.  
Hence we have proven~\eqref{eqn-attained-msrble}.  
\end{proof}

We now use Lemma~\ref{lem-scaled-metric-ratio} to prove a lower bound for the probability that $E_r(z)$ occurs for at least one small value of $r$. 

\begin{lem} \label{lem-shorter-annulus}
For each $q>1$, there exist parameters $\alpha \in (1/2,1)$ and $A>1$, depending only on $q$, such that for each compact set $K\subset U$ and each $\delta\in(0,1)$, we have
\eqb \label{eqn-shorter-annulus}
\inf_{z\in K} \BB P\left[ \text{$E_r(z)$ occurs for at least one $r\in [\ep^2   , \ep ] \cap \{2^{-k} : k\in\BB N\}$} \right] \geq 1 -  O_\ep(\ep^q )  .
\eqe  
\end{lem} 

To deal with condition~\ref{item-attained-long} in the definition of $E_r(z)$, we will use the following lower bound for the $D_h$-lengths of paths in a narrow Euclidean annulus, which is~\cite[Lemma 2.11]{gm-uniqueness} (note that the number $\frk c_r$ from~\cite{gm-uniqueness} is equal to $r^{\xi Q}$, see~\cite[Section 1.4]{gm-uniqueness}).

\begin{lem}[\cite{gm-uniqueness}] \label{lem-attained-long}
For each $S > s > 0$ and each $p\in (0,1)$, there exists $\alpha_* = \alpha_*(s,S,p) \in (1/2,1)$ such that for each $\alpha\in [\alpha_*,1)$, each $z\in\BB C$, and each $r > 0$, 
\eqb \label{eqn-attained-long} 
\BB P\left[ \inf\left\{ D_h\left( u , v ; \BB A_{\alpha r , r}(z) \right)  :  u , v \in \BB A_{\alpha  r , r }(z) , D_h(u,v) \geq s  r^{\xi Q} e^{\xi h_{\BB r}(z)} \right\} \geq S  r^{\xi Q} e^{\xi h_r(z)}  \right] \geq p . 
\eqe 
\end{lem}

\begin{proof}[Proof of Lemma~\ref{lem-shorter-annulus}]
By Lemma~\ref{lem-attained-msrble}, we can apply Lemma~\ref{lem-annulus-iterate} to find that there exists $p = p(q) \in (0,1)$ such that if~\eqref{eqn-shorter-annulus-show} just below holds, then~\eqref{eqn-shorter-annulus} holds:
\eqb \label{eqn-shorter-annulus-show}
\text{$\exists r_0 = r_0(K,\delta) > 0$ such that} \: \BB P[E_r(z)] \geq  p ,\quad \forall z\in K ,  \quad \forall r \in (0,r_0]  . 
\eqe
It therefore suffices to choose $\alpha$ and $A$ in a manner depending only on $p$ in such a way that~\eqref{eqn-shorter-annulus-show} holds. 

We first deal with condition~\ref{item-attained-long}.
By Axioms~\ref{item-metric-f} and~\refcoord, the fact that $D_h$ induces the Euclidean topology, and the scale and translation invariance of the law of $h$, modulo additive constant, we can find $S > s > 0$ depending on $ p$ such that for each sufficiently small $r>0$ (depending only on $K$), for each $z\in K$ it holds with probability at least $1 - (1- p)/4$ that  
\eqb
  D_h\left( \bdy \BB A_{3r/4,r}(z) , \bdy \BB A_{r/2 , 2r}(z)  \right)   \geq  s r^{\xi Q} e^{\xi h_r(z)} \quad \text{and} \quad 
\sup_{u,v\in \BB A_{3 r /4 , r}(z)} D_h\left( u ,v ; \BB A_{r/2,2r}(z)\right) \leq S r^{\xi Q} e^{\xi h_r(z)} .
\eqe 
Since $\phi$ is nearly linear at small scales, after possibly decreasing $s$ and increasing $S$ we can arrange that the same is true with $D^\phi_h$ in place of $D_h$. 
Since $ \BB A_{\alpha r , r}(z) \subset  \BB A_{3 r /4 , r}(z) $ for any choice of $\alpha \in [3/4,1)$, Lemma~\ref{lem-attained-long} with the above choice of $s$ and $S$ gives an $\alpha \in [3/4,1)$ depending on $ p$ such that for each sufficiently small $r > 0$, it holds for each $z\in K$ that the probability of condition~\ref{item-attained-long} in the definition of $E_r(z)$ is at least $1 - (1-p)/3$. 
 
By applying Axioms~\ref{item-metric-f} and~\refcoord\, as above, we can find $A > 1$ depending on $p$ such that for each sufficiently small $r > 0$, it holds for each $z\in K$ that the probability of condition~\ref{item-attained-around} in the definition of $E_r(z)$ is at least $1 - (1 -    p)/3$.
 
By Lemma~\ref{lem-scaled-metric-ratio} applied with $b = 1-\alpha$, for each sufficiently small $r > 0$, it holds for each $z\in K$ that the probability of condition~\ref{item-attained-dist} in the definition of $E_r(z)$ is at least $1-(1-p)/3$. Combining the three preceding paragraphs shows that~\eqref{eqn-shorter-annulus-show} holds.  
\end{proof}

\begin{lem} \label{lem-shorter-annulus-all}
There is a universal constant $q>1$ such that if $\alpha$ and $A$ are chosen as in Lemma~\ref{lem-shorter-annulus} for this choice of $q$, then for each compact set $K\subset U$ and each $\delta\in(0,1)$, it holds with probability tending to 1 as $\ep\rta 0$ that the following is true. For each $z\in K$, there exists $r   \in [\ep^2,\ep] \cap \{2^{-k} : k\in\BB N\}$ and $w\in \left(\frac{\ep^2 }{4}  \BB Z^2 \right) \cap B_\ep(K)$ such that $z\in B_{r/2}(w)$ and $E_r(w)  $ occurs. 
\end{lem}
\begin{proof} 
Upon choosing $q$ sufficiently large, this follows from Lemma~\ref{lem-shorter-annulus} and a union bound over all $w \in  \left(\frac{\ep^2}{4}  \BB Z^2 \right) \cap B_\ep(K)$. 
\end{proof}

\begin{figure}[t!]
 \begin{center}
\includegraphics[scale=1]{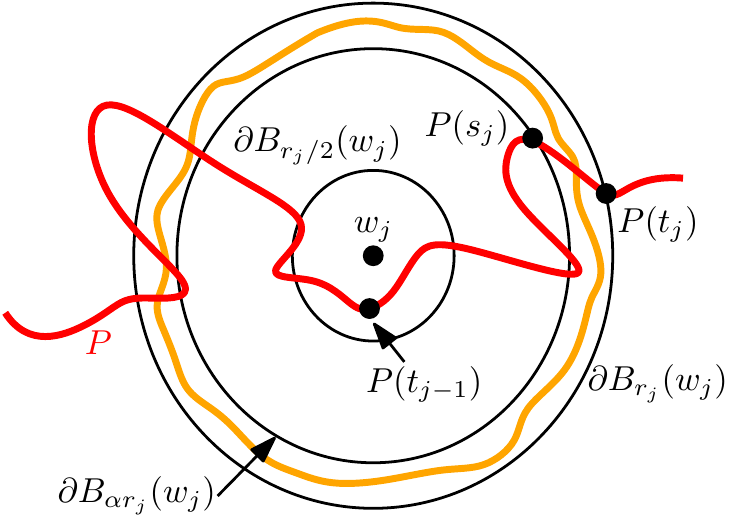}
\vspace{-0.01\textheight}
\caption{Illustration of the proof of Proposition~\ref{prop-attained}. The $D_h$-geodesic $P$ from $\BB z$ to $\BB w$ along with one of the balls $B_{r_j}(w_j)$ hit by $P$ for which $E_{r_j}(w_j)$ occurs are shown. 
The time $t_j$ is the first time after $t_{j-1}$ at which $P$ exits $B_{r_j}(w_j)$ and the time $s_j$ is the last time before $t_j$ at which $P$ hits $\bdy B_{\alpha r_j}(w_j)$. Condition~\ref{item-attained-dist} in the definition of $E_{r_j}(w_j)$ tells us that $D_h^\phi(P(s_j) , P(t_j)) \leq (1+\delta) (t_j - s_j)$. 
The orange path comes from condition~\ref{item-attained-around} in the definition of $E_{r_j}(w_j)$.
It has $D_h$-length is at most $A D_h(\bdy B_{\alpha r_j}(w_j) , \bdy B_{r_j}(w_j)) \leq A (t_j - s_j)$. 
Since $P$ is a $D_h$-geodesic and $P$ crosses this orange path both before time $t_{j-1}$ and after time $s_j$, it follows that $s_j - t_{j-1} \leq A (t_j - s_j)$. This allows us to show that the ``good" intervals $[s_j , t_j]$ occupy a uniformly positive fraction of the total $D_h$-length of $P$. This then allows us to show that $D_h^\phi(\BB z , \BB w) \leq C_\delta D_{h|_U}(\BB z ,\BB w)$ for a constant $C_\delta >0$ which is strictly smaller than $C_*$ if we assume that $C_*>1$ and $\delta$ is chosen to be sufficiently small. 
}\label{fig-attained}
\end{center}
\vspace{-1em}
\end{figure} 
 
\begin{proof}[Proof of Proposition~\ref{prop-attained}] 
See Figure~\ref{fig-attained} for an illustration of the proof.
\medskip

\noindent\textit{Step 1: setup.}
Let $\alpha$ and $A$ be chosen as in Lemma~\ref{lem-shorter-annulus-all}.
Also let 
\eqb \label{eqn-max-def}
C_* := \inf\left\{ C > 1 : \BB P\left[ \sup_{z,w\in U, z\not=w} \frac{D_h^\phi(z,w)}{D_{h|_U}(z,w)} \leq C  \right] = 1 \right\} .
\eqe
Proposition~\ref{prop-coord-bilip} implies that $C_* <\infty$. We want to show that $C_* \leq 1$.  

To this end, we will show that a.s.\ 
\eqb \label{eqn-lower-ratio}
D_h^\phi(\BB z , \BB w) \leq C_\delta D_{h|_U}(\BB z , \BB w) ,
\quad\forall \BB z,\BB  w \in U , \quad 
\text{where} \quad C_\delta := 1 + \delta     + \frac{A}{ A+1 }(C_* - 1 - \delta) . 
\eqe
If $C_* > 1$ and $\delta > 0 $ is chosen sufficiently small (depending on $C_*$ and $A$), then $C_\delta < C_*$. This contradicts the definition of $C_*$, so we infer that $C_* \leq 1$. 
It remains only to prove~\eqref{eqn-lower-ratio}. 
\medskip

\noindent\textit{Step 2: regularity event.} 
The idea of the proof of~\eqref{eqn-lower-ratio} is to use Lemma~\ref{lem-shorter-annulus-all} and conditions~\ref{item-attained-dist} and~\ref{item-attained-around} in the definition of $E_r(z)$ to show that if $P : [0,D_h(\BB z,\BB w)] \rta \BB C$ is a $D_h$-geodesic, then with high probability the following is true. We can cover a $1-\frac{A}{A+1}$-fraction of the interval $[0,D_h(\BB z,\BB w)]$ by intervals of the form $[s,t]$ such that $D_h^\phi(P(s) ,P(t)) \leq (1+\delta) (t-s)$. 
In order for the comparison of $D_h$-lengths and $D_h^\phi$-lengths to make sense, we need to make sure that our $D_h$-geodesic $P$ stays in $U$. We now introduce a regularity event on which we can force certain $D_h$-geodesics to stay in $U$. 

Fix a compact set $K\subset U$ and let $\zeta\in (0,1)$ be a small parameter which we will eventually send to zero.
By the continuity of $D_h$, we can find a small parameter $\rho \in (0,1)$ and a compact set $K'$ satisfying $K\subset K' \subset U$, depending on $K$ and $\zeta$, such that with probability at least $1-\zeta$, we have 
\eqb \label{eqn-attained-dist-reg}
D_h(\BB z , \BB w) \leq D_h(\BB z , \BB C\setminus K') ,\quad \forall \BB z , \BB w \in K \quad \text{with} \quad |\BB z-\BB w|  \leq \rho .
\eqe
If~\eqref{eqn-attained-dist-reg} holds, then $D_h(\BB z , \BB w) = D_{h|_U}(\BB z,\BB w)$ for each pair of points $\BB z , \BB w \in K $ with $|\BB z-\BB w|  \leq \rho $ and moreover every $D_h$-geodesic between two such points is contained in $K'$, so in particular is also a $D_{h|_U}$-geodesic. 

For $\ep > 0$, let $F_{K'}^\ep$ be the event that~\eqref{eqn-attained-dist-reg} holds and the event of Lemma~\ref{lem-shorter-annulus-all} occurs with the above choices of $\alpha,A,\delta $ and with $K'$ in place of $K$, so that $\BB P[F_{K'}^\ep] \geq 1 -\zeta - o_\ep(1)$.  
\medskip

\noindent\textit{Step 3: reducing to an estimate for nearby points.}
We claim that on $F_{K'}^\ep$, it is a.s.\ the case that
\eqb \label{eqn-attained-show}
D_h^\phi(\BB z , \BB w) \leq C_\delta D_{h|_U}(\BB z , \BB w) + o_\ep(1) ,\quad \forall \BB z , \BB w \in K \quad \text{with} \quad |\BB z-\BB w|  \leq \rho ,
\eqe
where the $o_\ep(1)$ is a random error which tends to zero in probability as $\ep\rta 0$, uniformly over all $\BB z,\BB w \in K$. 
Before proving~\eqref{eqn-attained-show}, we explain why it implies~\eqref{eqn-lower-ratio}.
Applying~\eqref{eqn-attained-show} and sending $\ep\rta 0$ shows that with probability at least $1-\zeta$, we have $D_h^\phi(\BB z , \BB w) \leq C_\delta D_{h|_U}(\BB z , \BB w)$
for each $ \BB z , \BB w \in K$ with $|\BB z-\BB w|  \leq \rho $. This implies that with probability least $1-\zeta$, the $D_h^\phi$-length of any path contained in $K$ is at most $C_\delta$ times its $D_{h|_U}$-length. Since $D_{h|_U}$ and $D_h^\phi$ are length metrics, sending $\zeta \rta 0$ and letting $K$ increase to all of $U$ gives~\eqref{eqn-lower-ratio}.  
\medskip

\noindent\textit{Step 4: decomposition of a $D_h$-geodesic into segments. }
Assume that $F_{K'}^\ep$ occurs, let $\BB z,\BB w \in K$ with $|\BB z-\BB w| \leq \rho$, and let $P$ be a $D_h$-geodesic from $\BB z$ to $\BB w$. As noted after~\eqref{eqn-attained-dist-reg}, we have $P\subset K'$. 
We will define several objects which depend on $P$ and $\ep$, but to lighten notation we will not make $P$ and $\ep$ explicit in the notation. See Figure~\ref{fig-attained} for an illustration of the definitions.

Let $t_0  = 0$ and inductively let $t_j$ for $j\in\BB N$ be the smallest time $t \geq t_{j-1}$ at which $P$ exits a Euclidean ball of the form $B_{r}(w)$ for $w\in \left(\frac{\ep^2 }{4} \BB Z^2 \right) \cap B_\ep(K)$ and $r\in [\ep^2 , \ep] \cap \{2^{-k} : k\in\BB N\}$ such that $P(t_{j-1}) \in B_{r/2}(w)$ and $E_r(w)$ occurs; or let $t_j = D_h(\BB z , \BB w )$ if no such $t$ exists.  
If $t_j  < D_h(\BB z , \BB w )$, let $w_j$ and $r_j$ be the corresponding values of $w$ and $r$. 
Also let $s_j$ be the last time before $t_j$ at which $P$ exits $B_{\alpha r_j}(w)$.
Note that $s_j \in [t_{j-1} , t_j]$ and $P([s_j , t_j]) \subset \ol{ \BB A_{\alpha r_j , r_j}(w_j)}$.

Define the indices
\eqb
\ul J := \max\left\{ j\in \BB N : |\BB z - P(t_{j-1} )| <  2 \ep \right\}
\quad \text{and} \quad
\ol J := \min\left\{j\in\BB N : |\BB w - P(t_{j+1} )|  <  2 \ep   \right\} .
\eqe
Since $r_j \leq \ep $ and $P(t_j) \in B_{r_j}(w_j)$ for each $j$, we have $\BB z , \BB w \notin B_{ r_j}(w_j)$ for $j \in [\ul J , \ol J]_{\BB Z}$. 
By the definition of $F_{K'}^\ep$, on this event we have $t_j  < D_h(\BB z,\BB w)$ and $|P(t_{j-1}) - P(t_j)| \leq 2 \ep $ whenever $|\BB w - P(t_{j-1})| \geq  \ep $. 
Therefore, on $F_{K'}^\ep$,  
\eqb \label{eqn-endpoint-contain}
P(t_{\ul J}) \in B_{4\ep }(\BB z) \quad \text{and} \quad P(t_{\ol J}) \in B_{4\ep }(\BB w) .
\eqe

Since $P$ is a $D_h$-geodesic, for $j\in [\ul J , \ol J]_{\BB Z}$ also $P|_{[s_j , t_j]}$ is a $D_h$-geodesic from $P(s_j) \in \bdy B_{\alpha r_j}(w_j)$ to $P(t_j) \in \bdy B_{r_j}(w_j)$.
By definition, this $D_h$-geodesic stays in $\ol{\BB A_{\alpha r_j , r_j}(w_j)}$. 
Combining this with condition~\ref{item-attained-dist} in the definition of $E_{r_j}(w_j)$ (applied with $u = P(s_j)$ and $v = P(t_j)$) and the definition~\eqref{eqn-max-def} of $C_*$, we obtain  
\eqb \label{eqn-increment-tilde-D}
D^\phi_h\left( P(s_j) , P(t_j)  \right) \leq (1+\delta) (t_j - s_j) 
\quad \text{and} \quad
D^\phi_h\left( P(t_{j-1}) , P(s_j) \right) \leq  C_* (s_j - t_{j-1}  ) ,\quad \forall j \in [\ul J ,\ol J]_{\BB Z}.
\eqe 
\medskip

\noindent\textit{Step 5: comparing $s_j - t_{j-1}$ to $t_j - s_j$.} 
In order to extract a non-trivial upper bound for $D_h^\phi(\BB z,\BB w)$ from~\eqref{eqn-increment-tilde-D}, we need to show that the ``good" intervals $[s_j,t_j]$ occupy a positive fraction of the total length of the time interval $[0,D_h(\BB z,\BB w)]$. To this end, we will now argue that $ s_j - t_{j-1}$ is not too much larger than $t_j - s_j$.
 
If $j \in [\ul J , \ol J]_{\BB Z}$, then since $r_j \leq \ep $ and $|P(t_j) -\BB z| \wedge |P(t_j) - \BB w| \geq 2\ep $, the geodesic $P$ must cross the annulus $\BB A_{\alpha r_j,r_j}(w_j)$ at least once before time $t_{j-1}$ and at least once after time $s_j$. 
By the definition of $E_{r_j}(w_j)$, there is a path disconnecting the inner and outer boundaries of this annulus with $D_h$-length at most $A D_h\left(\bdy B_{\alpha r_j}(w_j) , \bdy B_{  r_j}(w_j) \right)$. 
The geodesic $P$ must hit this path at least once before time $t_{j-1}$ and at least once after time $s_j$. 
Since $P$ is a geodesic and $P(s_j) \in \bdy B_{\alpha r_j}(w_j)$, $P(t_j) \in \bdy B_{ r_j}(w_j)$, it follows that 
\eqbn 
s_j - t_{j-1}  \leq A D_h\left(\bdy B_{\alpha r_j}(w_j) , \bdy B_{  r_j}(w_j) \right) \leq A (t_j - s_j) .
\eqen
Adding $A (s_j - t_{j-1})$ to both sides of this inequality, then dividing by $A+1$, gives
\eqb \label{eqn-increment-compare}
s_j - t_{j-1} \leq  \frac{A}{A+1} (t_j - t_{j-1}) .
\eqe
\medskip

\noindent\textit{Step 6: upper bound for $D^\phi_h$. }
As above, we assume that $F_{K'}^\ep$ occurs, we let $\BB z,\BB w\in K$ with $|\BB z-\BB w|\leq \rho$ and we let $P$ be the $D_h$-geodesic from $\BB z$ to $\BB w$. In the notation above, it holds for each $j\in [\ul J +1 , \ol J]_{\BB Z}$ that
\allb \label{eqn-attained-increment}
D_h^\phi\left(P(t_{j-1}) , P(t_j) \right)
&\leq    D^\phi_h\left( P(t_{j-1}) , P(s_j) \right) +  D^\phi_h\left( P(s_j) , P(t_j)  \right)    \quad  \text{(triangle inequality)} \notag \\
&\leq    C_*( s_j - t_{j-1})  +  (1+\delta) (t_j - s_j)  \quad  \text{(by \eqref{eqn-increment-tilde-D})} \notag \\
&=      (1+\delta) (t_j - t_{j-1} )  + (C_* - 1-\delta ) (s_j - t_{j-1})   \notag \\
&= C_\delta  (t_j - t_{j-1} )  \quad \text{(by \eqref{eqn-increment-compare} and the definition of $C_\delta$)} .
\alle
We now apply~\eqref{eqn-endpoint-contain} and sum the estimate~\eqref{eqn-attained-increment} to get
\allb \label{eqn-attained-sum}
D^\phi_h\left( B_{4 \ep \BB r}(\BB z) , B_{4 \ep \BB r}(\BB w)  \right) 
&\leq D_h^\phi\left( P(t_{\ul J}) , P(t_{\ol J}) \right) \quad \text{(by~\eqref{eqn-endpoint-contain})} \notag \\
&\leq  \sum_{j=\ul J+1}^{\ol J} D_h^\phi\left(P(t_{j-1}) , P(t_j) \right) \notag\\
&\leq C_\delta (t_{\ol J} - t_{\ul J}) \quad \text{(by~\eqref{eqn-attained-increment})} \notag\\ 
&\leq C_\delta D_{h|_U}(\BB z,\BB w) \quad \text{(since $P$ is a $D_{h|_U}$-geodesic)}.
\alle 

By the continuity of $(z,w) \mapsto D_h^\phi(z,w)$ and the triangle inequality, a.s.\  
\eqb \label{eqn-sup-distance}
D^\phi_h\left( \BB z , \BB w  \right)  \leq  D^\phi_h\left( B_{4 \ep \BB r}(\BB z) , B_{4 \ep \BB r}(\BB w)  \right)  + o_\ep(1) 
\eqe 
where the $o_\ep(1)$ tends to 0 in probability as $\ep\rta 0$, uniformly over all $\BB z,\BB w \in K $. Combining this with~\eqref{eqn-attained-sum} gives~\eqref{eqn-attained-show}. 
\end{proof}

\bibliography{cibiblong,cibib}
\bibliographystyle{hmralphaabbrv}

\end{document}